\documentclass[11pt]{amsart}
\usepackage[utf8x]{inputenc}
\usepackage[english]{babel}
\usepackage[T1]{fontenc}
 
\usepackage{graphicx}
\usepackage{caption}
\usepackage{subcaption}

\usepackage{amssymb,amsmath}
\usepackage{mathtools}

\usepackage{hyperref}
\usepackage[capitalise]{cleveref}

\usepackage{indentfirst}
\usepackage{enumerate,amsmath,amssymb, mathrsfs,mathtools}
\usepackage{appendix}
\usepackage{latexsym}
\usepackage{url}
\usepackage{color}
\usepackage{accents}
\usepackage{setspace}
\usepackage{pdfpages}
\usepackage{stmaryrd}

\usepackage[nobysame]{amsrefs}

\makeatletter
\@namedef{subjclassname@2020}{\textup{2020} Mathematics Subject Classification}
\makeatother

\usepackage[margin=2.5cm]{geometry}
\allowdisplaybreaks

\BibSpec{article}{%
	+{}{\PrintAuthors} {author}
	+{,}{ } {title}
	+{, }{\textit } {journal}
	+{}{ \parenthesize} {date}
		+{,  }{no. } {volume}
	+{,}{ } {pages}
	+{,}{ } {note}
}
\usepackage{bbm}

\newcommand{\id}{\operatorname{id}}
\def\Rho{\mbox{\textsf{P}}}
\newcommand{\tr}{\operatorname{tr}}
\newcommand{\Ric}{\operatorname{Ric}}

\ExplSyntaxOn

\NewExpandableDocumentCommand{\gobblefirst}{m}
{
	\tl_tail:n { #1 }
}

\ExplSyntaxOff
\BibSpec{arXiv}{%
  +{}{\PrintAuthors}{author}
  +{,}{ \textit}{title}
  +{}{ \parenthesize}{date}
  +{,}{ arXiv }{eprint}
}

\BibSpec{book}{%
	+{}{\PrintAuthors}  {author}
	+{. }{}{title}
	+{,}{ }{series}
	+{,}{ vol.~}{volume}
	+{. }{\textit}{publisher}
    +{}{ \parenthesize} {date}
	+{,}{ ISBN \gobblefirst}{isbn}
}
\BibSpec{collection.article}{%
	+{}{\PrintAuthors}{author}
	+{, }{}{title}
	+{, }{\textit}{booktitle}
	+{, }{ \DashPages}{pages}
	+{,}{ }{series}
	+{, }{}{volume}
	+{, }{\textit}{publisher}
	+{,}{ }{date}
}

\mathcode`l="8000
\begingroup
\makeatletter
\lccode`\~=`\l
\DeclareMathSymbol{\lsb@l}{\mathalpha}{letters}{`l}
\lowercase{\gdef~{\ifnum\the\mathgroup=\m@ne \ell \else \lsb@l \fi}}%
\endgroup

\def\XXint#1#2#3{{\setbox0=\hbox{$#1{#2#3}{\int}$ }
		\vcenter{\hbox{$#2#3$ }}\kern-.6\wd0}}

\newtheorem{proposition}{Proposition}

\newtheorem{definition}[proposition]{Definition}
\newtheorem{theorem}[proposition]{Theorem}

\newtheorem{remark}[proposition]{Remark}

\numberwithin{proposition}{section}
\numberwithin{equation}{section}

\title{Weyl structures for path geometries}
\author{Andreas \v Cap and Zhangwen Guo}

\thanks{This research was funded in whole or in part by the Austrian Science Fund
    (FWF): 10.55776/P33559 and 10.55776/Y963. For open access purposes, the authors
    have applied a CC BY public copyright license to any author-accepted manuscript
    version arising from this submission.  This article is based upon work from COST
    Action CaLISTA CA21109 supported by COST (European Cooperation in Science and
    Technology). https://www.cost.eu. }
\address{University of Vienna \newline \indent		Faculty of Mathematics  \newline \indent		Oskar-Morgenstern-Platz 1 \newline \indent 1090 Vienna,	Austria \newline\indent A.\v C.: \href{https://orcid.org/0000-0002-7745-3708}{https://orcid.org/0000-0002-7745-3708} \newline\indent Z.G.:		 \href{https://orcid.org/0009-0007-4027-4902}{https://orcid.org/0009-0007-4027-4902}        \newline \indent	\href{mailto:andreas.cap@univie.ac.at}{andreas.cap@univie.ac.at}	  \href{mailto:zhangwen.guo@univie.ac.atm}{zhangwen.guo@univie.ac.at}}
\subjclass{primary: 53C15; secondary: 53A40, 53B15}
\keywords{path geometry, distinguished connection, Weyl structure, tractor calculus, geometry of systems of ODE}

\begin{document}
\date{April 14, 2026}\onehalfspacing
\begin{abstract}
Path geometries provide a geometric encoding of systems of second order ODE, which serves as a model for the geometric theory of more general systems of ODE and for cone structures. They are an instance of the family of parabolic geometries, thus they are second order structures that are difficult to study using the usual tools of differential geometry. The general theory of parabolic geometries provides several efficient tools for the study of path geometries, but these use Cartan geometry methods and hence are not easily accessible. In this article, we build a bridge between these general methods and an elementary approach to path geometries. 

Motivated by the general theory of Weyl structures (but not using it), we first define a family of distinguished connections that is analogous to Webster-Tanaka connections in CR geometry. These are parametrized by (local) non-vanishing sections of a line bundle naturally associated to the geometry, and the dependence of this choice is described explicitly. We also discuss the Schouten tensor associated to such a choice and its dependence on the choice. We explain how these ingredients can be used to obtain an elementary approach to tractor calculus for path geometries and give examples of applications to the construction of invariant operators. 

A second major result that we prove is that in the case of path geometries, there is a smaller subclass of distinguished Weyl structures which does not seem to have an analog for any other type of parabolic geometries. This has interesting relations to the refinement of the de Rham complex induced by a path geometry via the machinery of BGG sequences. Again, all this is proved using elementary methods without reference to the general theory. 
\end{abstract}

    \maketitle 

\section{Introduction}

Path geometries were originally introduced by J.\ Douglas in \cite{Douglas} as a geometric description of systems of second order ODE which is independent of specific choices of coordinates. Classically, such a geometry on a smooth manifold $N$ is defined as a smooth family of paths (unparametrized curves) on $N$ such that for each point $x\in N$ and any line $\ell\subset T_xN$, there is a unique element in the family that contains $x$ and has tangent space $\ell$ in this point. To obtain a good description of such a geometry, one passes to the projectivzed tangent bundle $M:=\mathcal PTN$ of $N$, since the family of paths canonically lifts to a family on $\mathcal PTN$ with exactly one path through each point. As described in more detail in \S \ref{2.1} below, this lifted family can be equivalently described by a rank-one distribution $E$ which is transversal to the vertical subbundle $V$ of $\mathcal PTN\to N$ and contained in the so-called tautological subbundle $H$ of $\mathcal PTN$, which implies that $H=E\oplus V$. It tunrns out that this situation can be locally characterized by the (non-)integrability properties of $H=E\oplus V$, which then allows for a generalized concept, in which paths may only be defined in an open set of directions. 

In the language of systems of ODE, the equivalence problem for such structures was solved  by M.\ Fels in \cite{Fels}, exhibiting two fundamental invariants (one torsion and one curvature) which provide a complete obstruction to local isomorphism to the trivial geometry that corresponds to the equation $y''=0$. This was the basis for quite a lot of research on path geometries, see e.g. \cites{BGH,Grossman,Crampin-Saunders,CZ} which continue to be an active field. On the other hand, it was realized that the generalized version of path geometries falls into the class of parabolic geometries, see \cite{Book}, so in particular, they admit a canonical Cartan connection. This not only provides an alternative point of view on Fels' result, but also makes a number of general tools for the study of path geometries available. However, making these tools explicit for a specific type of structures is a non-trivial problem, and our article can be viewed as a contribution to this topic. At the same time, we want to stress the fact that while it motivates the developments, the general theory of parabolic geometries is not needed to follow the developments in this article. Most of our constructions use only elementary tools and the statements for which we appeal to the general theory can alternatively be checked directly. 

Our constructions are related to the general theory of Weyl structures for parabolic geometries, which was developed in \cite{Weyl}, see also Sections 5.1 and 5.2 of \cite{Book}. This provides a family of distinguished connections that come with additional distinguished geometric objects. They should be viewed as a (much more complicated) analog of the family of Webster-Tanaka connections, which play an important role in CR geometry. Let us briefly review Webster-Tanaka connections to put our results into perspective, more details can be found e.g.\ in Section 4 of \cite{Lee} or in \cite{Go-Gr}, see also Section 2.2 of \cite{Arman} and Section 5.2.12 of \cite{Book} for the non-integrable case. 

Recall that a CR structure of the relevant type on a manifold $M$ of odd dimension is given by a contact subbundle $H\subset TM$ which is endowed with an almost complex structure $J:H\to H$ that satisfies an integrability condition. Webster-Tanaka connections are associated to a choice of a (local or global) contact form  $\alpha\in\Omega^1(M)$, so for each $x$ in the domain of definition of $\alpha$, the kernel of $\alpha(x):T_xM\to\mathbb R$ is the contact subspace $H_x$. Such a contact form is well known to determine a Reeb vector field $r\in\mathfrak X(M)$ which has the property that $T_xM=H_x\oplus\mathbb R\cdot r(x)$ for any $x$. The Webster-Tanaka connection $\nabla^\alpha$ determined by $\alpha$ is then determined by the structural properties that $H$, $r$, and $J$ are parallel for $\nabla^\alpha$ plus normalization conditions on the torsion of $\nabla^\alpha$. 

In our situation, we also deal with a manifold $M$ of odd dimension $2n+1$, which is endowed with a subbundle $H$ of rank $n+1$, which is further decomposed as $H=E\oplus V$ where $E$ and $V$ have rank one and $n$, respectively. The canonical objects we construct are associated to a choice of scale, i.e.\ a local nowhere-vanishing section of the line bundle $E$. This has a natural interpretation as (locally) choosing parametrizations of the paths. More generally, we can start from a linear connection on the line bundle $E$. From this, we construct a subbundle of $TM$ which is complementary to $H$ and a linear connection on $TM$. (Note that for the choice of the complementary subbundle, there is much more freedom than in the case of Webster-Tanaka connections.) These data are uniquely characterized by the fact that the chosen section of $E$, the subbundle $V$ and the chosen complement all are parallel, and normalization conditions on the torsion, see Theorem \ref{thm:main} below for details. In case one starts with a connection on $E$, one has to require that $E$ is parallel and the induced connection is the chosen one.  

While existence and uniqueness of the connections we construct follow from the general theory of Weyl structures for parabolic geometries, our second main contribution has no known counterpart for any other type of parabolic geometries: In Section \ref{distinguished}, we identify a subclass of Weyl structures characterized by additional curvature conditions. Such special Weyl structures locally exist on any path geometry, so a stronger normalization is always possible locally. This subclass is related to a refinement of the de Rham complex on a manifold endowed with a path geometry, which is a very special instance of a Berstein-Gelfand-Gelfand sequence (BGG sequence) as introduced in \cite{CSS-BGG}. Again, we don't need input from the general theory but construct the part of the sequence that we need in a completely elementary way. 

In Section \ref{4}, we obtain an analog for path geometries of the elementary approach to tractor calculus for conformal and projective structures as introduced in \cite{BEG} building on classical work by T.\ Thomas, see \cite{Thomas}. We mainly discuss the standard tractor bundle and its canonical connection associated to a path geometry, from which more general tractor bundles and connections can be obtained via tensorial constructions. These bundles can then be used as a basis for the construction of invariant differential operators associated to path geometries. In the explicit formula for the tractor connection, another object canonically associated to a scale is needed, the so-called Schouten tensor $\Rho$, a $\binom02$-tensor field. We give an ad-hoc definition of this tensor, outlining why this coincides with the tensor provided by the general theory. Part of the approach is to derive how the objects associated to a scale behave under a change of scale, and we discuss this in Section \ref{4}, too.   

\section{Path geometries and Weyl structures}\label{2}
\subsection{(Generalized) path geometries}\label{2.1}
As briefly mentioned in the introduction, the classical definition of a \textit{path geometry} on a smooth manifold $N$ of dimension $n+1$ is a smooth family of paths (unparametrized curves) with exactly one curve through each point $x\in N$ in each direction defined by a line $\ell_x$ in $T_xN$. To describe this more explicitly and to clarify the notion of smoothness, it is advantageous to pass to the projectivized tangent bundle $\mathcal PTN$ of $N$, i.e.\ the space of all lines in tangent spaces of $N$. Via its tangent space in each point, any regularly parametrized curve in $N$ canonically lifts to $\mathcal PTN$ and this lifting is compatible with reparametrizations. Hence a family of paths on $N$ as above canonically lifts to a family of paths on $\mathcal PTN$ but now there is exactly one path through each point. The tangent spaces of the paths therefore determine a line in each tangent space of $\mathcal PTN$, so they define a rank one distribution $E\subset T\mathcal PTN$. Now the obvious condition is to call the family of paths on $N$ smooth if and only if $E$ is a smooth distribution on $\mathcal PTN$. Since any smooth distribution of rank one is involutive, one can recover the paths in the family as the projections of integral submanifolds for $E$. 

One can say more about the distribution $E$, however. Denote by $p:\mathcal PTN\to N$ the bundle projection. Let us also write $\ell_x$ for an element of $p^{-1}(\{x\})$, so this is a line in $T_xN$. Taking a tangent vector $\xi\in E_{\ell_x}\subset T_{\ell_x}\mathcal PTN$, the projected tangent vector $T_{\ell_x}p(\xi)$ by construction is tangent to the path in the family that goes through $x$ with tangent space $\ell_x$, so it is contained in $\ell_x$. The vectors with this property form the so-called \textit{tautological subbundle} $H\subset T\mathcal PTN$, which contains the vertical subbundle $V:=\ker(Tp)$ as a subbundle of corank one. Finally, for a non-zero vector $\xi\in E_{\ell_x}$, also $T_{\ell_x}p(\xi)$ is non-zero, so $E\cap V=\{0\}$ and hence $H=E\oplus V$. While $V$ is involutive by construction, it is well known that $H$ is maximally non-involutive. More precisely, denoting by $q:T\mathcal PTN\to T\mathcal PTN/H$ the projection, the map $\Gamma(E)\times\Gamma(V)\to \Gamma(T\mathcal PTN/H)$ defined by $(\xi,\eta)\mapsto q([\xi,\eta])$ is bilinear and hence induces a vector bundle homomorphism $E\otimes V\to T\mathcal PTN/H$, which turns out to be an isomorphism. Equivalently, if $\eta\in\Gamma(V)$ has the property that $[\xi,\eta](\ell_x)=0$ for one $\xi\in\Gamma(E)$ with $\xi(\ell_x)\neq 0$ then $\eta(\ell_x)=0$. In this form, there is an obvious ``abstract'' version of the structure as follows.

\begin{definition}\label{def2.1} 
Let $M$ be a smooth manifold of odd dimension $2n+1$. A (generalized) path geometry on $M$ is given by smooth distributions $E,V\subset TM$ of rank one and $n$, respectively, such that $E\cap V=\{0\}$, $V$ is involutive and for $H:=E\oplus V$ the bundle map $E\otimes V\to TM/H$ induced by the Lie bracket of vector fields is an isomorphism of vector bundles. We will denote this structure by $(M,E,V)$. 
\end{definition}

Given such a geometry $(M,E,V)$, involutivity of the distribution $V$ implies that there are local leaf spaces. So for sufficiently small open subsets $U\subset M$, one can find a surjective submersion $\psi:U\to N$ onto a smooth manifold $N$ of dimension $n+1$ such that for each $x\in U$ we get $\ker(T_x\psi)=V_x$, the fiber of $V$ at $x$. Since $E_x\cap V_x=\{0\}$, $T_x\psi(E_x)$ is a line in $T_{\psi(x)}N$ and this defines a lift $\Psi:U\to \mathcal PTN$ of $\psi$. It can be shown that the assumption of surjectivity of the map induced by the Lie bracket of vector fields then implies that $\Psi$ is a local diffeomorphism, which maps $V$ to the vertical subbundle of $\mathcal PTN\to N$ and $H$ to the tautological subbundle. Thus locally (also in the direction), a generalized path geometry is isomorphic to a path geometry in the classical sense. Hence we will drop the adjective ``generalized'' and just refer to $(M,E,V)$ as a path geometry in what follows.   

\begin{remark}\label{rem2.1}
There is a generalization of Definition \ref{def2.1}, in which instead of involutivity of $V$, one requires that the Lie bracket of two sections of $V$ is a section of $H=E\oplus V$. It turns out that if $2n+1\neq 5$, this condition acutally implies involutivity of $V$, see Proposition 4.4.4 of \cite{Book}. In dimension $5$, one obtains a more general notion which also leads to an additional fundamental invariant of the geometry. In fact, these more general geometries have interesting applications, in particular in the case that the distribution $V$ is generic, i.e.\ a so-called $(2,3,5)$-distribution, see \cite{An-Nurowski}. The construction of Weyl structures we present below can be generalized to all the structures in dimension $5$ satisfying the weaker condition.  We do not consider the more general situation in this article, since this would require special treatment in many places. 
\end{remark}

For a path geometry $(M,E,V)$, define $H:=E\oplus V$ and $Q:=TM/H$ and let  
$q:TM\twoheadrightarrow Q$ be the canonical quotient map. Recall that the Levi-bracket
\[\mathcal L:\Lambda^2 H\to Q\]
is the bundle map induced by the Lie bracket of vector fields, so it is characterized by $\mathcal L(\zeta_1(x),\zeta_2(x))=q([\zeta_1,\zeta_2](x))$. By Definition \ref{def2.1}, this vanishes on $\Lambda^2V\subset\Lambda^2H$ and restricts to an isomorphism $E\otimes V\to Q$. This has two important consequences: 
\begin{itemize}
    \item Together with sections $\xi\in\Gamma(E)$ and $\eta\in\Gamma(V)$, sections of the form $[\xi,\eta]$ span the tangent space in each point. 
    \item Having given linear connections on the subbundles $E$ and $V$ of $TM$ one can define a linear connection on the bundle $Q$ by requiring that $\mathcal L:E\otimes V\to Q$ is parallel.  
\end{itemize}

As discussed in the introduction, given a path geometry, we want to associate to a choice of scale a pair that consists of a linear connection on $TM$ and a subbundle of $TM$ that is complementary to $H$. For any such complement, $q:TM\to Q$ restricts to an isomorphism from the subbundle to $Q$, so the easiest way to describe it is via the inverse $\iota$ of this isomorphism. We will view this as an injective bundle map $\iota:Q\to TM$ such that $q\circ \iota=\id_Q$ and denote the complement by $\iota(Q)$. The second description we will use is via the projection $\Pi:TM\to H$ along $\iota(Q)$ that satisfies $\Pi|_H=\id_H$ and splits as $\Pi=\Pi_E\oplus\Pi_V$.      

Having given a linear connection $\nabla$ on $TM$, we can consider its torsion $\tau\in\Omega^2(M,TM)$ and its curvature $R\in \Omega^2(M,L(TM,TM))$, defined by the usual formulae for vector fields $\psi_i\in\mathfrak X(M)$. 
\begin{align*}\label{tors}
    \tau(\psi_1,\psi_2)&=\nabla_{\psi_1}\psi_2-\nabla_{\psi_2}\psi_1-[\psi_1,\psi_2]\\
    R(\psi_1,\psi_2)(\psi_3)&=\nabla_{\psi_1}\nabla_{\psi_2}\psi_3-\nabla_{\psi_2}\nabla_{\psi_1}\psi_3-\nabla_{[\psi_1,\psi_2]}\psi_3. 
\end{align*}
Having also given a decomposition $TM=E\oplus V\oplus\iota(Q)$ we can of course decompose $\tau$ and $R$ (as well as other tensor fields) accordingly. It will be crucial for the further developments that some of the resulting components do no depend on all the ingredients but need only partial information.

Recall finally, how the family of canonical objects should be parametrized. We will focus on the case that we start with a nowhere vanishing (local or global) section $\xi_0$ of $E$, which in this context we refer to as a (local) \textit{scale}. To a choice of scale, we associate a pair $(\nabla,\Pi)$, as discussed above, on the domain of definition of $\xi_0$. Since $E$ is a line bundle, a (local) scale can also be considered as a (local) trivialization of $E$. This trivialization determines a flat connection $\nabla^E$ on the restriction of $E$ to the domain of definition of $\xi_0$, which is characterized by the fact that $\nabla^E\xi_0=0$. Using a slight modification of the normalization condition, the construction can also be carried out starting from a general linear connection on $E$. We will not discuss this in detail but only provide some comments below, see in particular Remark \ref{rem:non-exact}. Now we are ready to formulate our first main result:

\begin{theorem}\label{thm:main}
    Let $(M,E,V)$ be a path geometry and let $\xi_0$ be a nowhere vanishing local smooth section of the subbundle $E\subset TM$ defined on an open subset $U\subset M$. Then there are a unique linear connection $\nabla$ on $TU$ and a unique decomposition $TU=E\oplus V\oplus \iota(Q)$ such that the following conditions are satisfied: 
    \begin{itemize}
        \item[(i)] The subbundles $E$, $V$ and $\iota(Q)$ of $TU$ are parallel for $\nabla$, and $\nabla\xi_0=0$. 
        \item[(ii)] The bundle map $\iota\circ\mathcal L:E\otimes V\to\iota(Q)$ is parallel for (the connections induced by) $\nabla$.
        \item[(iii)] The torsion $\tau$ of $\nabla$ satisfies $\tau(E,V)\subset \iota(Q)$ and $\tau(H,\iota(Q))\subset H$, and for sections $\eta_1,\eta_2\in\Gamma(V)$ and putting $\zeta_i:=\iota(\mathcal L(\xi_0,\eta_i))\in\Gamma(\iota(Q))$ we get 
        \begin{equation}\label{eq:norm}
            \mathcal L(\tau(\xi_0,\zeta_1),\eta_2)=2\mathcal L(\xi_0,\tau(\eta_1,\zeta_2)).
        \end{equation}
    \end{itemize}
\end{theorem}

\begin{remark}\label{rem:norm}
(1) Let us elaborate on the condition \eqref{eq:norm} a bit. First note that this condition makes sense since we also require that $\tau(H,\iota(Q))\subset H$. Moreover, the left hand side depends only on the $E$-component of $\tau(\xi_0,\zeta_1)$, while the right hand side depends only on the $V$-component of $\tau(\eta_1,\zeta_2)$. 

The map $\iota(\mathcal L(\xi_0,\_))$ defines an isomorphism $V\to\iota(Q)$ of vector bundles. Using this, we can interpret the component $V\times \iota(Q)\to V$ of $\tau$, as a bilinear bundle map $V\times V\to V$. Now \eqref{eq:norm} says that there actually is a section $\phi\in\Gamma(V^*)$  (characterized by the fact that the $E$-component of $\tau(\xi_0,\zeta_1)\in\Gamma (H)$ equals $2\phi(\eta_1)\xi_0$) such that this map is of the form $(\eta_1,\eta_2)\mapsto\phi(\eta_1)\eta_2$. So if one decomposes the above bundle map $V\times V\to V$ into trace parts and trace-free parts, this says that all components except one trace part vanish and that trace part is completely determined by another component of the torsion. 

(2) Observe that the uniqueness statement in Theorem \ref{thm:main} implies a compatibility with morphisms. Suppose that $(M,E,V)$ and $(\tilde M,\tilde E,\tilde V)$ are path geometries and that $f:M\to\tilde M$ is a local diffeomorphism such that, for each $x\in M$, $T_xf:T_xM\to T_{f(x)}\tilde M$ maps $E_x$ to $\tilde E_{f(x)}$ and $V_x$ to $\tilde V_{f(x)}$. Then for a (local) scale $\tilde\xi_0\in\Gamma(\tilde E)$, the pullback $f^*\tilde\xi_0$ is a (local) scale on $M$ and we can also pull back of the constituents of the Weyl structure on $\tilde M$ determined by $\tilde\xi_0$ to $M$. From the definitions, it also follows that pullbacks are compatible with $\mathcal L$ and with torsions of linear connections. Hence it follows readily that the pullbacks of the objects on $\tilde M$ satisfy all characterizing properties of the Weyl structure determined by $f^*\tilde\xi_0$. Hence any morphism pulls back (in an obvious sense) the Weyl structure determined by a scale to the Weyl structure determined by the pullback of the scale.     

(3) Although this is not needed formally, let us briefly explain (without discussing the background) why the Weyl structures we construct coincide with the ones obtained from the general theory of parabolic geometries. The latter Weyl connections by construction have properties (i) and (ii) of Theorem \ref{thm:main}, so in view of the uniqueness statement in that theorem, it remains to explain why they also satisfy (iii). The relation between the curvature and the torsion of a Weyl connection and the curvature $\kappa$ of the canonical Cartan connection is described in Theorem 5.2.9 and Lemma 5.2.10 (1) of \cite{Book}. This involves an additional object associated to a choice of scale, the \textit{Schouten tensor} $\Rho$ that will be discussed in \S \ref{4.4} below, which enters via its image under a certain tensorial operation denoted by $\partial(\Rho)$. Then all the conditions in (iii) apart from \eqref{eq:norm} readily follow from the fact that $\kappa$ is concentrated in homogeneities $\geq 2$, see Section 4.4.3 of \cite{Book}. The description of the homogeneity-two component of $\kappa$ from there also implies that the two components of the torsion that occur in \eqref{eq:norm} come from the corresponding components of $\partial(\Rho)$. Making $\partial(\Rho)$ explicit immediately shows that \eqref{eq:norm} is equivalent to the fact that for $\xi\in\Gamma(E)$ and $\eta\in\Gamma(V)$, one gets $\Rho(\xi,\eta)+2\Rho(\eta,\xi)=0$. See Remark \ref{rem:Rho} below for an explanation on how to obtain this from known properties of $\kappa$. 
\end{remark}

The proof of Theorem \ref{thm:main} will be carried out in two steps, which are formulated as propositions.  Observe that a path geometry can be restricted to open subsets without problems, so in the proof, it suffices to consider the case that $\xi_0$ is globally defined . In the first step, we construct the subbundle $\iota(Q)$ and the restriction of $\nabla$ to a \textit{partial connection}, i.e.\ to an operator $\Gamma(H)\times\mathfrak X(M)\to\mathfrak X(M)$ which has the defining properties of a linear connection (linearity over smooth functions in the first variable and Leibniz rule in the second variable). Notice that the property of a subbundle of $TM$ being parallel makes sense in the setting of partial connections. 

\begin{proposition}\label{prop:homog1}
    Let $(M,E,V)$ be a path geometry and let $\xi_0\in\Gamma(E)$ be a nowhere vanishing section. Then we can uniquely define a projection $\Pi=\Pi_E\oplus\Pi_V:TM\to H$ and partial connections $\nabla^E$, $\nabla^V$ and $\nabla^{\iota(Q)}$ on the indicated bundles by the following properties for $\xi=f\xi_0\in \Gamma(E)$ where $f\in C^\infty(M)$ and $\eta,\eta_1,\eta_2\in\Gamma(V)$.
\begin{equation}\label{homog1-props}
\begin{aligned}
\begin{cases}
\nabla^E\xi&=df\otimes\xi_0\\
\mathcal L(\xi_0,\nabla^V_{\xi}\eta)=\mathcal L(\xi_0,\Pi_V([\xi,\eta]))&=\frac 1 2 q([\xi,[\xi_0,\eta]])\\
\Pi_E([\xi,\eta])&=-df(\eta)\xi_0\\
\mathcal L(\xi_0,\nabla^V_{\eta_1}\eta_2)&=q([\eta_1,[\xi_0,\eta_2]])\\
\nabla^{\iota(Q)}\iota(\mathcal L(\xi,\eta))&=\iota(\mathcal L(\nabla^E\xi,\eta)+\mathcal L(\xi,\nabla^V\eta)).
\end{cases}
\end{aligned}
\end{equation}
The projection $\Pi$ and the partial connection $\nabla$ on $TM$ obtained from $\nabla^E$, $\nabla^V$ and $\nabla^{\iota(Q)}$ via the decomposition $TM=E\oplus V\oplus\iota(Q)$  satisfy conditions (i) and (ii) and the first two properties in condition (iii) of Theorem \ref{thm:main}, which in turn determine them uniquely.  
\end{proposition}
\begin{proof}
Of course, the first line of \eqref{homog1-props} uniquely defines a (partial) linear connection on $E$ such that $\nabla^E\xi_0=0$. Since $\mathcal L(\xi_0,\_):V\to Q$ is bijective, the second and fourth line of \eqref{homog1-props} uniquely define sections $\nabla^V_\xi\eta$, and $\nabla^V_{\eta_1}\eta_2$ of $V$. Hence we obtain an operator $\nabla^V:\Gamma(H)\times\Gamma(V)\to\Gamma(V)$. Using standard properties of the Lie bracket, one readily verifies that this operator in linear over smooth functions in the first variable and satisfies a Leibniz rule in the second variable (and the factor $\tfrac12$ in the second line is crucial for this). Hence we have defined partial connections on $E$ and $V$. 

For a vector field $\zeta\in\mathfrak X(M)$, the section $q(\zeta)\in\Gamma(Q)$ determines a unique $\eta\in\Gamma(V)$ such that $q(\zeta)=\mathcal L(\xi_0,\eta)$ and hence $\zeta-[\xi_0,\eta]\in\Gamma(H)$. Now we define an operator $\Pi:\mathfrak X(M)\to\Gamma(H)$ by $\Pi(\zeta):=\nabla^V_{\xi_0}\eta+(\zeta-[\xi_0,\eta])$. The Leibniz rule for $\nabla^V$ then readily implies that $\Pi$ is linear over smooth functions and hence defines a bundle map $TM\to H$. Since for $\zeta\in\Gamma(H)$ we get $\eta=0$ and hence $\Pi(\zeta)=\zeta$, this is a projection as required. This clearly is consistent with the second and third lines of \eqref{homog1-props} and determined by them. Thus we have constructed all claimed objects. 

Condition (i) in Theorem \ref{thm:main} on the partial connection underlying $\nabla$ is equivalent to the fact that it is the direct sum of partial connections on the subbundles $E$, $V$ and $\iota(Q)$, with the connection on $E$ defined by the first line of \eqref{homog1-props}. Knowing this, condition (ii) in Theorem \ref{thm:main} is equivalent to the last line in \eqref{homog1-props}. The first property in condition (iii) of Theorem \ref{thm:main} exactly says that $\Pi([\xi,\eta])=\nabla^V_\xi\eta-\nabla^E_\eta\xi$. Hence it depends only on the underlying partial connection and is equivalent to the first equality in the second line and the third line in \eqref{homog1-props}. 

So it remains to analyze the second property in condition (iii) of Theorem \ref{thm:main}. For this we consider $\zeta:=\iota(\mathcal L(\xi_0,\eta))$, so we already know that $q(\nabla^{\iota(Q)}\zeta)=\mathcal L(\xi_0,\nabla^V\eta)$. While we don't know derivatives in $\iota(Q)$-directions yet, they will preserve the subbundle $H$ and hence will not contribute to the image of the torsion under $q$. Hence $0=q(\tau(\xi,\zeta))$ amounts to $\mathcal L(\xi_0,\nabla^V_\xi\eta)=q([\xi,\zeta])$. But from the description of $\Pi([\xi_0,\eta])$ we conclude that $\zeta=[\xi_0,\eta]-\nabla^V_{\xi_0}\eta$, so $q([\xi,\zeta])=q([\xi,[\xi_0,\eta]])-\mathcal L(\xi,\nabla^V_{\xi_0}\eta)$, and of course we can swap $\xi$ and $\xi_0$ in the last term. Thus we are left with $q([\xi,[\xi_0,\eta]])=2\mathcal L(\xi_0,\nabla^V_\xi\eta)$, which forces the rest of the second line of \eqref{homog1-props} and this equation obviously holds for any $\xi$ if and only if it holds for $\xi=\xi_0$. Taking another section $\eta_1\in\Gamma(V)$ we see in the same way that $0=q(\tau(\eta_1,\zeta))$ amounts to $\mathcal L(\xi_0,\nabla^V_{\eta_1}\eta)=q([\eta_1,[\xi_0,\eta]])$ and hence to the fourth line in \eqref{homog1-props}. 
\end{proof}

We remark that the projection $\Pi_E$ associated to $\xi_0\in\Gamma(E)$ can also be interpreted as a one-form $\alpha\in\Omega^1(M)$ characterized by $\Pi_E=\alpha\xi_0$, i.e.\ $\Pi_E(\zeta)=\alpha(\zeta)\xi_0$ for any vector field $\zeta\in\mathfrak X(M)$. This form will play a major role in the discussion of distinguished Weyl structures below. 

To complete the proof of Theorem \ref{thm:main}, it remains to define covariant derivatives in directions contained in $\iota(Q)$, which is done in the following result. 
\begin{proposition}\label{prop:homog2}
    Let $(M,E,V)$ be a path geometry, $\xi_0\in\Gamma(E)$ be a nowhere vanishing section, $\Pi:TM\to H$ the projection from Proposition \ref{prop:homog1} and $\alpha\in\Omega^1(M)$ the one-form such that $\Pi_E=\alpha\xi_0$. Then the partial connections $\nabla^E$, $\nabla^V$ and $\nabla^{\iota(Q)}$ from this proposition can be uniquely extended to linear connections by requiring that $\nabla^E\xi_0=0$ and the following conditions for $\eta,\tilde\eta\in\Gamma(V)$ and $\zeta=\iota(\mathcal L(\xi_0,\eta))\in\Gamma(\iota(Q))$:
    \begin{equation}\label{homog2-props}
    \begin{aligned}
     \nabla^V_{\zeta}\tilde\eta&=\Pi_V([\zeta,\tilde\eta])+\tfrac12\alpha([\xi_0,[\xi_0,\tilde\eta]])\eta\\
     \nabla^{\iota(Q)}\zeta &=\iota(\mathcal L(\xi_0,\nabla^V\eta)).    
    \end{aligned}
    \end{equation}
    Combining the connections to a linear connection $\nabla$ on $TM$, the resulting pair $(\nabla,\Pi)$ satisfies all conditions from Theorem \ref{thm:main} and is uniquely determined by them. 
\end{proposition}
\begin{proof}
  For a smooth function $f$ on $M$, we get $f\zeta=\iota(\mathcal L(\xi_0,f\eta))$ and together with $\Pi_V(\zeta)=0$, this readily implies that the right hand side of the first line of \eqref{homog2-props} is linear over smooth functions in $\zeta$. The properties of the Lie bracket also imply that it satisfies the Leibniz rule in the other variable, so it uniquely specifies an extension of the partial connection $\nabla^V$ to a linear connection. The second line of \eqref{homog2-props} then defines an extension of $\nabla^{\iota(Q)}$, while $\nabla^E\xi_0=0$ trivially extends $\nabla^E$. Thus it remains to verify the claims about the conditions from Theorem \ref{thm:main} and we only have to consider property \eqref{eq:norm}. 

So let us take $\eta_1,\eta_2\in\Gamma(V)$ and put $\zeta_i=\iota(\mathcal L(\xi_0,\eta_i))$ for $i=1,2$. Knowing that the subbundles $E$ and $\iota(Q)$ are parallel for $\nabla$, the left hand side of \eqref{eq:norm} by definition is given by $-\mathcal L(\Pi_E([\xi_0,\zeta_1]),\eta_2)=-\alpha([\xi_0,\zeta_1])\zeta_2$. Likewise, we can expand the right hand side to see that \eqref{eq:norm} is equivalent to  
$$
\alpha([\xi_0,\zeta_1])\mathcal L(\xi_0,\eta_2)=2\mathcal L(\xi_0,\nabla^V_{\zeta_2}\eta_1+\Pi_V([\eta_1,\zeta_2])).   
$$
  Now we observe that $\zeta_1-[\xi_0,\eta_1]\in\Gamma(V)$, whence by the third line in \eqref{homog1-props} we can replace $\zeta_1$ by $[\xi_0,\eta_1]$ in the left hand side, so this is satisfied for the extension defined by the first line in \eqref{homog2-props}. Since all other ingredients are known, this obviously in turn uniquely determines $\nabla^V_{\zeta_2}\eta_1$. 
\end{proof}

\begin{remark}\label{rem:non-exact}
As briefly mentioned above, there is a slightly larger family of distinguished connections. This more general version starts from an arbitrary linear connection $\nabla^E$ on the bundle $E$ and then constructs $\Pi$ and connections $\nabla^V$ on $V$ and $\nabla^{\iota(Q)}$ on $\iota(Q)$ to obtain a linear connection $\nabla$ on $TM$. Apart from the obvious change that $\nabla$ should restrict to the given connection $\nabla^E$ on $E$, the only part in Theorem \ref{thm:main} that has to be modified is the condition in formula \eqref{eq:norm}. Here one has to involve the curvature $R^E\in\Omega^2(M,L(E,E))$ of $\nabla^E$ and require that for $\xi\in\Gamma(E)$ and $\eta_1,\eta_2\in\Gamma(V)$ one has
$$
\mathcal L\Big(R^E(\xi,\eta_1)\xi+
\tau(\xi,\iota(\mathcal L(\xi,\eta_1))),\eta_2\Big)=2\,\mathcal L\Big(\xi,\tau(\eta_1,\iota(\mathcal L(\xi,\eta_2)))\Big).
$$
The proof of this more general version is completely parallel to what we have done above. 
\end{remark}

Before we move to the study of special Weyl structures, we can prove several additional properties of the torsion and curvature of the connections obtained in Theorem \ref{thm:main}. 
\begin{theorem}\label{tors-curv}
  Let $(M,E,V)$ be a path geometry, $\xi_0\in\Gamma(E)$ a nowhere vanishing section, $(\nabla,\Pi)$ the data associated to $\xi_0$ in Theorem \ref{thm:main} and let $\tau$ and $R$ be the torsion and the curvature of $\nabla$. Then we have $\tau(V,V)=0$ and $\tau(TM,\iota(Q))\subset H$. Moreover, denoting by $\alpha\in\Omega^1(M)$ the form charaterized by $\Pi_E=\alpha\xi_0$, we get for $\eta_1,\eta_2,\eta_3\in\Gamma(V)$ and $\zeta_i=\iota(\mathcal L(\xi_0,\eta_i))$ for $i=1,2,3$ the following:
\begin{equation}\label{eq:tors-curv}
\begin{aligned}
\begin{cases}
\tau(\eta_1,\zeta_2)&=-\tfrac12\alpha([\xi_0,\zeta_1])\eta_2-\alpha([\eta_1,\zeta_2])\xi_0\\ 
R(\xi_0,\eta_1)(\eta_2)&=-\tfrac12\alpha([\xi_0,\zeta_1])\,\eta_2-\alpha([\xi_0,\zeta_2])\,\eta_1\\
R(\eta_1,\eta_2)(\eta_3)&=-\alpha([\eta_1,\zeta_3])\,\eta_2+\alpha([\eta_2,\zeta_3])\,\eta_1.
\end{cases}
\end{aligned}
\end{equation}  
\end{theorem}
\begin{remark}\label{rem:curvature}
   We have already observed in Remark \ref{rem:norm} that the component $V\times\iota(Q)\to V$ of $\tau$ is ``pure trace''. Similarly, the other two equations in \eqref{eq:tors-curv} in particular imply that several components of the curvatures vanish identically. In particular, fixing $\xi_0$, the second line defines a section of $V^*\otimes V^*\otimes V$ and the contraction into the first $V^*$-entry is given by $\eta_2\mapsto -\frac{2n+1}{2}\alpha([\xi_0,\zeta_2])$, so we exactly get an inclusion of a trace component into the curvature. Similarly, the Ricci-type contraction in the third line maps $(\eta_2,\eta_3)$ to $(n-1)\alpha([\eta_2,\zeta_3])$, so again we get an inclusion of a trace component. Note also that $\alpha(\xi_0)$ is constant while $\alpha(\eta_i)=\alpha(\zeta_i)=0$, so the right hand sides of $\eqref{eq:tors-curv}$ can all be expressed in terms of $d\alpha$. 
\end{remark}
\begin{proof}[Proof of Theorem \ref{tors-curv}]
    By the Jacobi identity $0=[\eta_1,[\xi_0,\eta_2]]-[\eta_2,[\xi_0,\eta_1]]-[\xi_0,[\eta_1,\eta_2]]$. By involutivity, we get $[\eta_1,\eta_2]\in\Gamma(V)$ so applying $q$ to this equation and using the fourth line of \eqref{homog1-props}, we get $\mathcal L(\xi_0,\tau(\eta_1,\eta_2))=0$. Since $\tau(V,V)\subset V$ by definition, this implies $\tau(V,V)=0$. 
    
    Since we know that $\tau(\eta_1,\zeta_2)=\nabla_{\eta_1}\zeta_2-\nabla_{\zeta_2}\eta_1-[\eta_1,\zeta_2]$ lies in $\Gamma(H)$, it has to be given by $-\nabla_{\zeta_2}\eta_1-\Pi_V([\eta_1,\zeta_2])-\Pi_E([\eta_1,\zeta_2])$. By the first line of \eqref{homog2-props}, the first two summands give $-\tfrac12\alpha([\xi_0,[\xi_0,\eta_1]])\eta_2$. As we have observed already we can replace $[\xi_0,\eta_1]$ with $\zeta_1$ without changing the result, so the first line of \eqref{eq:tors-curv} follows readily. 

    The fourth line of \eqref{homog1-props} shows that  
    $\mathcal L(\xi_0,\nabla_{\eta_1}\nabla_{\eta_2}\eta_3)=q([\eta_1,[\xi_0,\nabla_{\eta_2}\eta_3]])$.
    It also shows that $q([\xi_0,\nabla_{\eta_2}\eta_3])=q([\eta_2,[\xi_0,\eta_3]])$, so $[\xi_0,\nabla_{\eta_2}\eta_3]-[\eta_2,[\xi_0,\eta_3]]+\Pi_E([\eta_2,[\xi_0,\eta_3]])\in\Gamma(V)$ by the third line of \eqref{homog1-props}. By involutivity of $V$, we conclude that
    $$
    \mathcal L(\xi_0,\nabla_{\eta_1}\nabla_{\eta_2}\eta_3)=q([\eta_1,[\eta_2,[\xi_0,\eta_3]]]) + \alpha([\eta_2,[\xi_0,\eta_3]])\mathcal L(\xi_0,\eta_1). 
    $$
    Subtracting the same term with $\eta_1$ and $\eta_2$ exchanged as well as $\mathcal L(\xi_0,\nabla_{[\eta_1,\eta_2]}\eta_3)=q([[\eta_1,\eta_2],[\xi_0,\eta_3]])$, the first terms in the right hand side cancel. Since $R(\eta_1,\eta_2)(\eta_3)$ is a section of $V$, this implies the last line of \eqref{eq:tors-curv}. 

    To deal with the second line of \eqref{eq:tors-curv}, we first observe that both sides of the claimed identity are linear over smooth functions in $\eta_1$ and $\eta_2$, so it suffices to verify them locally for the elements of a frame. Locally around each point, we can fix a hypersurface $\tilde M\subset M$ whose tangent spaces are transversal to $E$, then choose a local frame of $V$ along $\tilde M$ and extend it by parallel transport along $\xi_0$ to a local frame defined on an open subset. The frame elements then are parallel in $E$-directions, so we can prove our identity assuming that $\nabla_{\xi_0}\eta_i=0$ for $i=1,2$. This also implies that $\Pi_V([\xi_0,\eta_i])=0$ for $i=1,2$ and hence $\zeta_i=[\xi_0,\eta_i]$ and $\nabla_{\xi_0}\zeta_i=0$ for $i=1,2$. Hence the second line of \eqref{homog2-props} shows that 
    \begin{equation}\label{tech1}
    -\mathcal L(\xi_0,\nabla_{[\xi_0,\eta_1]}\eta_2)=-\mathcal L(\xi_0,\nabla_{\zeta_1}\eta_2)=-q([\xi_0,\Pi_V([\zeta_1,\eta_2]))-\tfrac12\alpha([\xi_0,[\xi_0,\eta_2]])\mathcal L(\xi_0,\eta_1). 
    \end{equation}
    Writing $\Pi_{-2}:=\id-\Pi$, we can write $\Pi_V([\zeta_1,\eta_2])=[\zeta_1,\eta_2]-\Pi_E([\zeta_1,\eta_2])-\Pi_{-2}([\zeta_1,\eta_2])$, and plugging into $q([\xi_0,\_])$ we can ignore the $\Pi_E$-term. Since $q(\tau(V,\iota(Q)))=0$, we get $\Pi_{-2}([\zeta_1,\eta_2])=-\nabla_{\eta_2}\zeta_1$ and since $q(\tau(E,\iota(Q)))=0$, we get $q([\xi_0,\nabla_{\eta_2}\zeta_1])=q(\nabla_{\xi_0}\nabla_{\eta_2}\zeta_1)=\mathcal L(\xi_0,\nabla_{\xi_0}\nabla_{\eta_2}\eta_1)$, where in the last step we have used that $\iota\circ \mathcal L(\xi_0,\_)$ is parallel. Hence we conclude that 
    \begin{equation}\label{tech2}
      -q([\xi_0,\Pi_V([\zeta_1,\eta_2]))=-q([\xi_0,[\zeta_1,\eta_2]])-\mathcal L(\xi_0,\nabla_{\xi_0}\nabla_{\eta_2}\eta_1). 
    \end{equation}
    To compute $\mathcal L(\xi_0,R(\xi_0,\eta_1)\eta_2)$, we have to add $\mathcal L(\xi_0,\nabla_{\xi_0}\nabla_{\eta_1}\eta_2)$ to \eqref{tech1}, and in view of $\tau(V,V)=0$, this adds up with the last term of \eqref{tech2} to $\mathcal L(\xi_0,\nabla_{\xi_0}[\eta_1,\eta_2])=\tfrac12q([\xi_0,[\xi_0,[\eta_1,\eta_2]]])$. Hence 
    \begin{equation}\label{tech3}
      \mathcal L(\xi_0,R(\xi_0,\eta_1)\eta_2)=\tfrac12q([\xi_0,[\xi_0,[\eta_1,\eta_2]]])-q([\xi_0,[\zeta_1,\eta_2]]) -\tfrac12\alpha([\xi_0,[\xi_0,\eta_2]])\mathcal L(\xi_0,\eta_1).
    \end{equation}
    The Jacobi identity and $\zeta_i=[\xi_0,\eta_i]$ readily show that the first two terms in the right hand side add up to $-\frac12 q([\xi_0,[\zeta_1,\eta_2]])-\frac12q([\xi_0,[\zeta_2,\eta_1]])=-\frac12 q([[\xi_0,\zeta_1],\eta_2])-\frac12q([[\xi_0,\zeta_2],\eta_1])$. Now in the defining equation for $\tau(\xi_0,\zeta_i)$ the derivative terms both vanish, so $-[\xi_0,\zeta_i]=\tau(\xi_0,\zeta_i)\in\Gamma(H)$. Forming the bracket with $\eta_j$ and applying $q$, only the $E$-component matters and this is given by $\alpha([\xi_0,\zeta_i])\xi_0$.  
\end{proof}    

\section{Distinguished scales}\label{distinguished}
Our second aim is to single out a distinguished subclass of the Weyl structures constructed in Theorem \ref{thm:main}. The deeper reason for the existence of this subclass is that for a path geometry $(M,E,V)$ there is a natural (linear) differential operator $D$ acting on sections of the dual bundle $E^*$ to $E$. Since nowhere vanishing sections of $E$ are equivalent to nowhere vanishing sections of $E^*$, such a subclass is given by sections for which the dual lies in the kernel of $D$. This situation also occurs for other types of parabolic geometries, in particular for conformal and projective structures, but with a completely different flavor. In these cases, the relevant operator $D$ is the first operator in a BGG sequence (a "first BGG operator") and hence is overdetermined, so existence of non-trivial sections in the kernel is a restriction on the geometry. Indeed, sections in the kernel of the relevant operator correspond to Einstein metrics in a conformal class and to Ricci-flat connections in a projective class, respectively, and hence do not exist in general. 

For path geometries, the relevant operator is the second operator in the BGG sequence that refines the (scalar) de Rham complex (which is not available for conformal and projective structures). As we shall see below, involutivity of $V$ implies that locally there always are nowhere-vanishing sections in the kernel of the relevant operator, so the resulting distinguished Weyl structures do always exist locally. While the proof of existence of such sections turns out to be quite simple, the result is rather surprising from the point of view of BGG sequences and to our knowledge does not follow from general BGG-theory. 

\subsection{The scalar BGG sequence}\label{BGG}
We first discuss the relevant refinement of the de Rham sequence (in low degrees) in an elementary way. For a path geometry $(M,E,V)$, the subbundle $H=E\oplus V\subset TM$ gives rise to nested subbundles in the bundles of differential forms. In degree one, we obtain a subbundle of $T^*M$, whose fiber in a point $x$ consists of those linear maps $T_xM\to\mathbb R$ that vanish on $H_x\subset T_xM$. This can be identified with the dual bundle $Q^*$ to $Q=TM/H$ and the quotient by this subbundle is isomorphic to $H^*=E^*\oplus V^*$. Hence we get a short exact sequence of the form $0\to Q^*\to T^*M\to E^*\oplus V^*\to 0$.  

In higher degrees, the subbundles are similarly defined as maps that vanish if sufficiently many of their entries are in $H$. In particular, one always obtains $\Lambda^k H^*$ as a quotient of $\Lambda^kT^*M$ by the subbundle consisting of multilinear maps that vanish if one of their entries lies in $H$. The next smaller subbundle is obtained from maps that vanish if two of their entries lie in $H$ and so on. If $k\leq n$, then the smallest non-zero subbundle is obtained from maps that vanish if all their entries are from $H$ and this can be identified with $\Lambda^kQ^*$. 

The general theory of BGG sequences is based on maps $\partial^*:\Lambda^kT^*M\to\Lambda^{k-1}T^*M$ such that $\partial^*\circ\partial^*=0$, which in turn can be obtained from the fact that $T^*M$ can be naturally made into a bundle of nilpotent Lie algebras. Since we are mainly interested in degree one, we don't go into this but derive the necessary facts directly and only in the degrees in which we need them. 

\begin{proposition}\label{prop:BGG}
    Let $(M,E,V)$ be a path geometry and let $\alpha_0\in\Gamma(H^*)$ be a smooth section.

    (1) There is a unique $\alpha\in\Omega^1(M)$ such that $\alpha|_H=\alpha_0$ and such that $d\alpha|_{E\times V}=0$. Mapping $\alpha_0$ to $\alpha$ defines a linear first order differential operator $L:\Gamma(H^*)\to\Omega^1(M)$. 

    (2) Mapping $\alpha_0$ to $dL(\alpha_0)|_{V\times V}$ defines a linear first order operator $\Gamma(H^*)\to\Gamma(\Lambda^2V^*)$, which vanishes if $\alpha_0\in\Gamma(E^*)\subset\Gamma(H^*)$, i.e.\ $\alpha_0|_V=0$. 

    (3) For $\alpha_0\in\Gamma(E^*)$, the form $dL(\alpha_0)$ vanishes upon insertion of two sections of $H$ and hence determines a well defined section $D(\alpha_0)$ of $H^*\otimes Q^*\cong E^*\otimes Q^*\oplus V^*\otimes Q^*$. Viewing $Q^*\otimes V^*$ as $E^*\otimes V^*\otimes V^*$, $D(\alpha_0)$ actually has values in the subbundle $E^*\otimes S^2V^*$. 

    (4) If $\alpha_0\in\Gamma(E^*)$ satisfies $D(\alpha_0)=0$, then $dL(\alpha_0)=0$.  
\end{proposition}
\begin{proof}
    (1) Suppose that $\beta\in\Omega^1(M)$ satisfies $\beta|_H=0$. Then by definition for $\xi\in\Gamma(E)$ and $\eta\in\Gamma(V)$, we get $d\beta(\xi,\eta)=-\beta([\xi,\eta])$. Viewing $\beta$ as a section of $Q^*$, this means that $d\beta(\xi,\eta)=-\beta(\mathcal L(\xi,\eta))$, so any section of $E^*\otimes V^*$ can be uniquely written as $d\beta$ for such a section $\beta$. Given $\alpha_0\in\Gamma(H^*)$, we can choose any $\hat\alpha\in\Omega^1(M)$ such that $\hat\alpha|_H=\alpha_0$ (for example using an appropriate local coframe) and then put $L(\alpha_0):=\hat\alpha+\beta$, where $\beta\in\Gamma(Q^*)\subset\Omega^1(M)$ is characterized by $d\beta|_{E\times V}=-d\hat\alpha|_{E\times V}$. 

    (2) It is clear that $\alpha_0\mapsto dL(\alpha_0)|_{V\times V}$ defines a differential operator of order at most two. Putting $\alpha:=L(\alpha_0)$, involutivity of $V$ and the global formula for the exterior derivative readily imply that $d\alpha|_{V\times V}$ depends only on $\alpha|_V=\alpha_0|_V$. This shows that the operator is of first order and that it vanishes if $\alpha_0|_V=0$. 

    (3) For $\alpha=L(\alpha_0)$, $d\alpha$ vanishes on $E\times V$ by (1) and on $V\times V$ by (2). Thus it vanishes upon insertion of two sections of $H$ and hence its restriction to $H\times TM$ naturally descends to $H\times Q$, which implies existence of $D(\alpha_0)$. For the last claim, we have to show that for $\eta_1,\eta_2\in\Gamma(V)$ and $\xi\in\Gamma(E)$, we get $d\alpha(\eta_1,[\xi_0,\eta_2])=d\alpha(\eta_2,[\xi_0,\eta_1])$. This readly follows from expanding the equation $0=d(d\alpha)(\xi,\eta_1,\eta_2)$ taking into account that $d\alpha|_{H\times H}=0$ and $[\eta_1,\eta_2]\in\Gamma(H)$.  

    (4) The assumption that $D(\alpha_0)=0$ means that $\alpha:=L(\alpha_0)$ has the property that $d\alpha$ vanishes upon insertion of one section of $H$, and to prove the claim, it suffices to show that for any $\xi\in\Gamma(E)$ and $\eta_1,\eta_2\in\Gamma(H)$, we have $d\alpha([\xi,\eta_1],[\xi,\eta_2])=0$. But this readily follows from expanding $0=d(d\alpha)(\xi,\eta_1,[\xi,\eta_2])$ according to the global formula for the exterior derivative, which apart from the negative of the term in question only involves terms in which one section of $H$ is inserted into $d\alpha$. 
\end{proof}

\begin{remark}
(1)  Observe that for $\dim(M)=5$, part (2) of this proposition depends on the assumption that the subbundle $V$ is involutive, compare to Remark \ref{rem2.1}. Indeed, if one only assumes that the bracket of two sections of $V$ always is a section of $H$, then the Lie bracket induces a tensorial map $V\otimes V\to E$ which captures the obstruction to involutivity of $V$. For $\alpha_0\in\Gamma(E^*)$, $\alpha=L(\alpha_0)$ and $\eta_1,\eta_2\in\Gamma(V)$, one then gets $d\alpha(\eta_1,\eta_2)=-\alpha_0([\eta_1,\eta_2])$, so this equivalently encodes that tensorial map.  

(2) As a simple special case of the theory of BGG sequences, the considerations in part (3) and (4) of Proposition \ref{prop:BGG} can be extended to general sections of $H^*$ with a bit of additional care. The general version of the operator $D$ then has three components with values in $\Lambda^2V^*$ (this comes from part (2) of Proposition \ref{prop:BGG}), $E^*\otimes Q^*$ (which is easily seen to be well defined) and $E^*\otimes S^2V^*\subset V^*\otimes Q^*$ (which is a bit more complicated to obtain). For $f\in C^\infty(M,\mathbb R)$, one then defines $Df\in\Gamma(H^*)$ as $df|_H$ and shows that the operator $L$ induces an isomorphis between $\text{ker}(D)/\text{im}(D)$ and the first de Rham cohomology $H^1(M,\mathbb R)=\text{ker}(d)/\text{im}(d)$. 
  \end{remark}

  \subsection{Application to Weyl structures}\label{dist-Weyl}
  The relevance of the above considerations for Weyl structures is rather obvious. Given a local nowhere-vanishing section $\xi_0$ of $E$, we obtain a local nowhere-vanishing section $\alpha_0$ of $E^*$ characterized by $\alpha_0(\xi_0)\equiv 1$ and vice versa. Hence in Theorem \ref{thm:main} one could also use $\alpha_0$ as the starting point for the construction of a Weyl structure, so one is naturally led to the following definition. 

\begin{definition}
A Weyl structure $(\nabla,\Pi)$ as in Theorem \ref{thm:main} is called \textit{distinguished} if and only if the corresponding section $\alpha_0\in\Gamma(E^*)$ satisfies $D(\alpha_0)=0$ or equivalently (in view of part (4) or Proposition \ref{prop:BGG}) $d(L(\alpha_0))=0$. 
\end{definition}

Now we can prove our second main result, which states that distinguished Weyl structures always exist locally and describes their special properties. 

\begin{theorem}\label{thm:dist}
    Let $(M,E,V)$ be a path geometry for which the subbundle $V$ is involutive. 

    (1) Locally around each point $x\in M$, there exists a distinguished Weyl structure. 

    (2) Let $(\nabla,\Pi)$ be a distinguished Weyl structure and let $\tau$ and $R$ be the torsion and the curvature of $\nabla$. Then in addition to the properties listed in Theorems \ref{thm:main} and \ref{tors-curv}, then the components $\tau|_{V\times\iota(Q)}$ and $R|_{\Lambda^2H\times V}$ vanish.
\end{theorem}
\begin{proof}
    (1) Involutivity of $V$ implies that there is an open neighborhood $U$ of $x$ in $M$ and a surjective submersion $\psi:U\to N$ onto a smooth manifold $N$ such that $\text{ker}(T_y\psi)=V_y$ for any $y\in U$. Now $T_x\psi(E_x)$ is a line $\ell\subset T_{\psi(x)N}$ and we can choose a smooth function $f:N\to\mathbb R$ such that $df(\psi(x))$ has a non-zero restriction to $\ell$. Define $\alpha:=\psi^*(df)\in\Omega^1(U)$ and $\alpha_0=\alpha|_E\in\Gamma(E^*)$. Then by construction $\alpha|_V=0$ and $d\alpha=0$, which by parts (1) and (3) of Proposition \ref{prop:BGG} implies that $\alpha=L(\alpha_0)$ and $D(\alpha_0)=0$. Since $\alpha_0(x)$ is nonzero by construction, it defines a distinguished Weyl structure on a sufficiently small open neighborhood of $x$. 

    (2) Suppose that $\alpha_0\in\Gamma(E^*)$ defines a distinguished Weyl structure $(\nabla,\Pi)$ corresponding to $\xi_0\in\Gamma(E)$, and let $\alpha\in\Omega^1(M)$ be characterized by $\Pi_E=\alpha\xi_0$. Then the third line of \eqref{homog1-props} says that for $\eta\in\Gamma(V)$, we get $0=\alpha([\xi_0,\eta])$. Since $\alpha(\xi_0)\equiv 1$ and $\alpha(\eta)=0$, this equals $-d\alpha(\xi_0,\eta)$, so $\alpha=L(\alpha_0)$ by part (1) of Proposition \ref{prop:BGG}. Moreover, part (4) of that Propostion shows that $d\alpha=0$. Now the right hand sides of all three lines in \eqref{eq:tors-curv} can be rewritten in terms of $d\alpha$, showing that the left hand sides of these equations all vanish. This implies the claimed properties of torsion and curvature. 
\end{proof}

\begin{remark}\label{rem:dist}
    The construction of distinguished scales in the proof of part (1) actually produces all distinguished scales locally. If $\alpha_0\in\Gamma(E^*)$ is a distinguished scale then by definition $\alpha=L(\alpha_0)$ satisfies $\alpha|_V=0$ and $d\alpha=0$, so in prticular, sections of $V$ insert trivially into both $\alpha$ and $d\alpha$. It is well known that $\alpha$ can locally be written as the pullback of a one-form on a local leaf space. The latter form of course has to be closed, too, and hence locally exact. 
\end{remark}

\section{Tractor calculus}\label{4}
As an application of the description of Weyl structures, we now develop the elementary approach to tractor calculus for path geometries. This is parallel to the constructions in \cite{BEG} for conformal and projective structures and to the one for CR structures (based on Webster-Tanaka connections) in \cite{Go-Gr}. While this approach produces the tractor bundles and tractor connections obtained from the general theory of parabolic geometries, the description of tractors can be used without reference to this theory. The standard tractor bundle $\mathcal T M$ is defined via its space of sections. Given a choice of scale, a section is defined via a triple of sections of bundles constructed from $E$, $V$, and $TM/H$ and one defines explicitly how the same section is decomposed with respect to a different scale, which provides a definition of sections as equivalence classes. The second step then is to define a linear connection on $\mathcal TM$. This is defined in the identification with triples determined by a scale using the Weyl structure determined by that scale and additional data, the Schouten tensor to be discussed below. It follows then from the general theory that a different choice of scale leads to the same linear connection on $\mathcal{T}M$, so this \textit{standard tractor connection} is canonically associated to the path geometry. There is an alternative elementary approach to this by explicit computation. To apply this approach, one needs to understand how the objects involved in the construction change under a change of scale, so we provide the necessary results. For the Weyl structure itself, this can be deduced from the construction in Section \ref{2} by straightforward computations. This would also be possible for the Schouten tensor, but the computations get very involved there. Hence in this case, we just sketch how to deduce the explicit formulae for the tensor and their behavior under a change of scale from the general theory. 

\subsection{Changing the scale}\label{4.1} 
We start by analyzing how the objects associated to a choice of scale in Section \ref{2} change if one changes the scale, which of course is of interest for any application of Weyl structures. Given a non-vanishing section $\xi_0\in\Gamma(E)$, any other such section can be written as $\hat\xi_0=e^f\xi_0$ for some smooth function $f$ and we will follow the usual convention that objects associated to $\hat\xi_0$ are indicated by hats. 

Before we can do that, we need a few observations on $1$-forms. Given any one-form $\gamma\in\Omega^1(M)$, we can restrict $\gamma$ to $E$ and $V$ to obtain sections $\gamma^E$ and $\gamma^V$ of $E^*$ and $V^*$, respectively. Via the isomorphism $Q\cong E\otimes V$, we can also view $\gamma^E$ as defining a bundle map $Q\to V$ and $\gamma^V$ as defining a bundle map $Q\to E$. We will denote these by just plugging elements sections of $Q$ into $\gamma^E$ and $\gamma^V$, so in particular for $\xi\in\Gamma(E)$ and $\eta\in\Gamma(V)$, we get $\gamma^E(\mathcal L(\xi,\eta))=\gamma(\xi)\eta$ and $\gamma^V(\mathcal L(\xi,\eta))=\gamma(\eta)\xi$. This notation is justified by the fact that these operations behave in the usual way with respect to covariant derivatives by a Weyl connection, so for $\psi\in\mathfrak X(M)$ and $\zeta\in\Gamma(Q)$, we get $\nabla_\psi \gamma^E(\zeta)=(\nabla_\psi\gamma^E)(\zeta)-\gamma^E(\nabla_\psi \zeta)$ and so on. In order to not make the notation too complicated, we will not indicate the components in the case that we just plug a vector field into a one-form, so we will usually write $\gamma(\xi)$ and $\gamma(\eta)$ and not $\gamma^E(\xi)$ and $\gamma^V(\eta)$ for $\xi\in\Gamma(E)$ and $V\in\Gamma(V)$. 

A choice of Weyl structure gives us a complementary component $\gamma^Q\in\Gamma(Q^*)$, which is defined by $\gamma^Q(\psi):=\gamma(\psi-\Pi(\psi))$ so it corresponds to the restriction of $\gamma$ to $\iota(Q)\subset T^*M$. In case of exact one-forms, we will use the notation $d^Ef$, $d^Vf$ and $d^Qf$ for the components of $df$.  

\begin{proposition}\label{prop:Weyl-transf}
Let $(M,E,V)$ be a path geometry, $\xi_0\in\Gamma(E)$ a scale, and put $\hat\xi_0:=e^f\xi_0$ for a smooth function $f$. Then using the notation introduced above, we get the following. 

(i) The projections $\Pi$ and $\hat\Pi$ defined by $\xi_0$ and $\hat\xi_0$ are for any $\psi\in\mathfrak X(M)$ related by
\begin{equation}\label{eq:Pi-trans}
        \hat\Pi_E(\psi)=\Pi_E(\psi)+d^Vf(q(\psi)) \qquad \hat\Pi_V(\psi)=\Pi_V(\psi)+\tfrac12d^Ef(q(\psi)). 
\end{equation}
Consequently, $\hat\alpha(\psi)=e^{-f}\big(\alpha(\psi)+\alpha_0(d^Vf(q(\psi))\big)$ and for $\xi\in\Gamma(E)$ and $\eta\in\Gamma(V)$, we get \begin{equation}\label{hat-zeta}
 \hat\iota(\mathcal L(\xi,\eta))=\iota(\mathcal{L}(\xi,\eta))-df(\eta)\xi-\tfrac12 df(\xi)\eta.   
\end{equation}  

(ii) For $\psi\in\mathfrak X(M)$, $\xi,\xi_1,\xi_2\in\Gamma(E)$, $\eta,\eta_1,\eta_2\in\Gamma(V)$ and $\zeta_j:=\iota(\mathcal L(\xi_j,\eta_j))$, the Weyl connections on $E$ and $V$ induced by $\xi_0$ and $\hat\xi_0$ are related by  
\begin{equation}\label{eq:ne-trans}
    \hat\nabla_\psi \xi=\nabla_\psi \xi-df(\psi)\xi. 
\end{equation}
\begin{equation}\label{eq:nv-trans}
    \begin{cases}
      \hat\nabla_\xi\eta=\nabla_\xi\eta+\tfrac12 df(\xi)\eta \\
      \hat\nabla_{\eta_1}\eta_2=\nabla_{\eta_1}\eta_2+df(\eta_1)\eta_2+df(\eta_2)\eta_1 \\
      \hat\nabla_{\zeta_1}\eta_2=\nabla_{\zeta_1}\eta_2+\tfrac12 df(\xi_1)df(\eta_1)\eta_2+df(\xi_1)df(\eta_2)\eta_1-d^Qf(\mathcal L(\xi_1,\eta_2))\eta_1. 
    \end{cases}
\end{equation}
\begin{equation}\label{eq:nq-trans}
    \begin{cases}
        \begin{aligned}
            \hat\nabla_\xi\zeta_1=&\nabla_\xi\zeta_1-\tfrac12df(\xi)\zeta_1+\big((\nabla_\xi df)(\eta_1)-\tfrac12df(\xi)df(\eta_1)\big)\xi_1\\+&\tfrac12\big((\nabla_\xi df)(\xi_1)+df(\xi)df(\xi_1)\big)\eta_1
        \end{aligned}\\
        \begin{aligned}
            \hat\nabla_\eta\zeta_1=&\nabla_\eta\zeta_1+df(\eta_1)\iota(\mathcal{L}(\xi_1,\eta))+\big((\nabla_\eta df)(\eta_1)-2df(\eta)df(\eta_1)\big)\xi_1\\+&\tfrac12\big((\nabla_\eta df)(\xi_1)+df(\eta)df(\xi_1)\big)\eta_1
        \end{aligned}\\
        \begin{aligned}
            \hat\nabla_{\zeta_1}\zeta_2=&\nabla_{\zeta_1}\zeta_2+\big(\tfrac12df(\xi_1)df(\eta_1)-df(\zeta_1)\big)\zeta_2+\big(df(\xi_2)df(\eta_2)-df(\zeta_2)\big)\zeta_1\\ 
            +&(\nabla_{\zeta_1}df)(\eta_2)\xi_2+\big(df(\zeta_2)df(\eta_1)-\tfrac32df(\xi_2)df(\eta_1)df(\eta_2)\big)\xi_1\\
            +&\tfrac12\big((\nabla_{\zeta_1} df)(\xi_2)+df(\zeta_1)df(\xi_2)\big)\eta_2.
        \end{aligned}
    \end{cases}
\end{equation}
\end{proposition}

\begin{proof}
Since the difference $\hat\Pi(\psi)-\Pi(\psi)$ depends only on $q(\psi)$, it suffices to verify
\eqref{eq:Pi-trans} for $\psi=[\xi,\eta]$ with $\xi\in\Gamma(E)$ and $\eta\in\Gamma(V)$, and one may even take $\xi=\xi_0$. For this case, all claims in (i) can be easily verified by direct computations.  Similarly, \eqref{eq:ne-trans} and the first two lines of \eqref{eq:nv-trans} can be easily verified directly. 

For the last line in \eqref{eq:nv-trans}, the argument is slightly more complicated. Since the difference between two connections is tensorial in all arguments, we may without loss of generality assume that $\nabla_{\xi_0}\eta_i=0$ for $i=1,2$, and this in particular implies that $\zeta_1=[\xi_0,\eta_1]$. But then we can use the definition of curvature to express $\nabla_{[\xi_0,\eta_1]}\eta_2$ in terms of $R(\xi_0,\eta_1)(\eta_2)$ and of iterated covariant derivatives in directions of $\xi_0$ and $\eta_1$. The same works for the Weyl connection $\hat\nabla$ and its curvature $\hat R$. The relevant components of $R$ and $\hat R$ can be expressed in terms of $\alpha$ and $\hat\alpha$ by Theorem \ref{tors-curv} from which we also understand the behavior under a change of scale. This suffices to verify the claimed formula by a direct computation. 

For \eqref{eq:nq-trans}, one first obtains $\hat\nabla_\psi \zeta_1=\hat\nabla_\psi\hat\zeta_1+\hat\nabla_\psi(df(\eta_1)\xi_1)+\tfrac12\hat\nabla_\psi(df(\xi_1)\eta_1)$ from \eqref{hat-zeta}. The first term in the right hand side, by definition can be expanded as $\hat\iota(\mathcal L(\hat\nabla_{\psi}\xi_1,\eta_1)+\mathcal L(\xi_1,\hat\nabla_\psi\eta_1))$, which again can be expanded using \eqref{hat-zeta}, which already leads to some cancellation. Apart from $\nabla_\psi\zeta_1$, all the remaining terms contain only quantities that we have already computed in \eqref{eq:ne-trans} or \eqref{eq:nv-trans}    
\end{proof}

\begin{remark}\label{rem4.2}
(1) This result also provides important information on how to relate our explicit description to the general theory of Weyl structures developed in Chapter 5 of \cite{Book}. In the general approach, the change of Weyl structure is described by a quantity $\Upsilon$, which in the case of path geometries decomposes into $\Upsilon^E\in\Gamma(E^*)$, $\Upsilon^V\in\Gamma(V^*)$ and $\Upsilon^Q\in\Gamma(Q^*)$. Comparing our explicit formulae to the ones provided by the general theory, one concludes that for $\xi\in\Gamma(E)$ and $\eta\in\Gamma(V)$, one obtains $\Upsilon^E(\xi)=-\frac12df(\xi)$, $\Upsilon^V(\eta)=df(\eta)$ and $\Upsilon^Q(\mathcal{L}(\xi,\eta))=\frac34df(\xi)df(\eta)-df(\iota(\mathcal{L}(\xi,\eta)))$.  

(2) When starting from a distinguished scale, the curvature property $R(E,V)V=0$ simplifies the computation for the last line in \eqref{eq:nv-trans}.
\end{remark}

\subsection{Density bundles}\label{4.2}
The importance of natural line bundles for the construction of invariant operations is well known in conformal and CR geometry and from the general theory of parabolic geometries. Similarly as in CR geometries, the bundles of this type associated to a path geometry are parametrized by two real numbers. Since we already have the line bundle $E$ available, we have to construct one more basic line bundle. 

Consider the line bundle $E\otimes\Lambda^n V^*$. Forming its square (i.e.\ the tensor product with itself) $(E\otimes\Lambda^nV^*)^2$, we obtain a line bundle which is orientable and hence trivial. Thus we can form any real power of this bundle and using this we define:
\begin{definition}\label{def:L}
For a (generalized) path geometry $(M,E,V)$, we define the line bundle $L\to M$ as $L:=(E\otimes\Lambda^nV^*)^{\frac{2}{2n+4}}$. 
\end{definition}
Any Weyl structure induces linear connections on $E$ and $V$ and hence on $E^*\otimes\Lambda^nV$ and thus on $L$. It turns out that, together with $E$, the line bundle $L$ can be used to obtain all natural line bundles on a path geometry by tensorial constructions. In particular, all these line bundles inherit a Weyl connection and the behavior of these connections under a change of scale can be immediately deduced from $\eqref{eq:ne-trans}$ and the following result. 

\begin{proposition}\label{prop:L-trans}
    Let $(M,E,V)$ be a (generalized) path geometry, let $\xi_0\in\Gamma(E)$ be a scale and put $\hat\xi_0:=e^f\xi_0$ for a smooth function $f$. Take $\xi\in\Gamma(E)$, $\eta\in\Gamma(V)$ and put $\zeta:=\iota(\mathcal L (\xi,\eta))$ and let $\nabla$ and $\hat\nabla$ denote Weyl connections with respect to $\xi_0$ and $\hat\xi_0$, respectively. Then for $\rho\in\Gamma(L)$ we get 
    \begin{equation}\label{n-trans-L}
        \begin{cases}
            \hat\nabla_\xi \rho=\nabla_\xi\rho -\tfrac12 df(\xi)\rho\\
            \hat\nabla_{\eta}\rho=\nabla_{\eta}\rho-df(\eta)\rho\\
            \hat\nabla_{\zeta}\rho=\nabla_{\zeta}\rho-\frac12df(\xi)df(\eta)\rho. 
        \end{cases} 
    \end{equation} 
\end{proposition}
\begin{proof}
   Let $\omega\in\Gamma(\Lambda^n V^*)$ be a local nowhere vanishing section and choose a local frame $\eta_1,\dots,\eta_n$ for $V$ such that $\omega(\eta_1,\dots,\eta_n)\equiv 1$. For a vector field $\psi\in\mathfrak X(M)$, we thus have 
   $$(\nabla_{\psi}\omega)(\eta_1,\dots,\eta_n)=-\textstyle\sum_i\omega(\eta_1,\dots,\nabla_\psi\eta_i,\dots,\eta_n)$$ 
   and likewise for $\hat\nabla$. Using this, the claimed formulae can be easily verified directly from the results in Proposition \ref{prop:Weyl-transf}, taking into account that we can always expand a section $\eta\in\Gamma(V)$ in terms of the $\eta_i$.  
\end{proof}

\begin{remark}
(1) Together with \eqref{eq:ne-trans}, this result readily shows that on the bundle $E^*\otimes L$, the covariant derivatives in $V$-directions with respect to all Weyl connections agree. Hence this line bundle inherits a canonical partial linear connection in $V$-directions. Similarly, the bundle $E^*\otimes L\otimes L$ inherits a canonical partial linear connection in $E$-directions. This corresponds to the fact that these two line bundles are relative tractor bundles (corresponding to two different parabolic subalgebras), compare to \cite{Relative}. 

(2) The line bundle $L$ also plays an important role in other parts of the theory of path geometries. As shown in \cite{Guo_Fef}, the complement of the zero section in $L$ inherits a natural almost Grassmannian structure from the given path geometry via a so-called Fefferman construction, see also \cite{Crampin-Saunders}.   
\end{remark}

\subsection{Standard tractors}\label{4.3}
As indicated above, we will define a vector bundle $\mathcal TM\to M$ of rank $n+2$ via its space of sections. Given a choice $\xi_0\in\Gamma(E)$ of scale, we define such a section as a triple consisting of sections $\nu\in\Gamma(V\otimes L)$, $\rho\in\Gamma(L)$ and $\tau\in\Gamma(E^*\otimes L)$. We will use a vectorial notation for such sections, sometimes indicating the scale defining the triple as a subscript. More frequently, we will use a second scale $\hat\xi_0=e^f\xi_0$ and denote quantities referring to that scale by hats. To define the relations of triples with respect to different scales, we have to recall some facts about one-forms from Section \ref{4.1}. To simplify notation, we will simply insert tensor products of vector fields with sections of line bundles into differential forms. So for example given $\rho\in\Gamma(L)$, we can form $d^Ef\otimes \rho\in\Gamma(E^*\otimes L)$, while for $\nu\in\Gamma(V\otimes L)$, we can form $df(\nu)=d^Vf(\nu)\in\Gamma(L)$ as well as $d^Qf(\nu)\in\Gamma(E^*\otimes L)$. The latter section is characterized by $d^Qf(\eta\otimes\rho)(\xi)=df(\iota(\mathcal L(\xi,\eta)))\rho$. Using these observations, we now define
\begin{equation}\label{tract-change}
\begin{pmatrix} \nu\\ \rho \\ \tau \end{pmatrix}_{\hat\xi_0}=\widehat{\begin{pmatrix} \nu\\ \rho \\ \tau \end{pmatrix}}=\begin{pmatrix} \nu\\ \rho-df(\nu) \\ \tau+\frac12d^Ef\otimes\rho+\frac12d^Ef\otimes df(\nu)-d^Qf(\nu)\end{pmatrix}_{\xi_0}
\end{equation}
It is easy to see that this defines an equivalence relation on the set of pairs of the form $(\xi_0,(\nu,\rho,\tau))$. On the set of equivalence classes, there is a well defined multiplication by smooth functions (which affects only the $(\nu,\rho,\tau)$ component) and this makes the set of equivalence classes into a sheaf of modules over the sheaf of smooth real valued functions. Using local frames of  $V\otimes L$, $L$, and $L\otimes E^*$, one easily concludes that this sheaf is locally free of constant rank $n+2$. Hence standard arguments show that there is a unique vector bundle $\mathcal{T}M\to M$ of rank $n+2$ which has this this sheaf as its sheaf  of sections. The detour via sections is necessary since the value of right hand side of \eqref{tract-change} in a point $x\in M$ does not only depend on the values of the ingredients in $x$ but also on derivatives (in the $\xi_0$-variable). This reflects the fact that, similar to tractor bundles in conformal and projective geometry, this tractor bundle is a second order object. As observed in part (2) of Remark \ref{rem:norm} a morphism of path geometries can be used to pull back scales, so by the above construction, it also defines a pullback of sections of the standard tractor bundles associated to the two path geometries. This in turn induces a vector bundle homomorphism between the tractor bundles. As for conformal and projective structures, morphisms of path geometries are determined locally around a point $x$ by their two-jet in $x$, but in general not by their one-jet in $x$. And indeed the map between the fibers of the tractor bundles over $x$ depends on the two-jet and not only on the one-jet of the morphism in $x$. 

It is easy, however, to obtain relations between the standard tactor bundle $\mathcal TM$ and simpler vector bundles. From \eqref{tract-change} it readily follows that inclusion of the bottom component gives rise to a well defined inclusion $E^*\otimes L\hookrightarrow \mathcal TM$, which we can view as defining a line subbundle $\mathcal T^1M\cong E^*\otimes L$. Similarly, mapping to the top component defines a surjective bundle map $\mathcal{T}M\to V\otimes L$, whose kernel is a subbundle $\mathcal T^0M\subset\mathcal TM$ of rank two, such that $\mathcal T^1M\subset\mathcal{T}^0M$ and $\mathcal T^0M/\mathcal T^1M\cong L$. In this language, the above definition of $\mathcal TM$ can be rephrased as the statement that any choice of scale gives rise to an isomomorphism $\mathcal{T}M\cong (E^*\otimes L)\oplus L\oplus (V\otimes L)$ and \eqref{tract-change} describes how this isomorphism depends on this choice. 

Having the vector bundle $\mathcal{T}M$ at hand, we can define more general tractor bundles via tensorial operations. For each of the resulting bundles a choice of scale defines an induced isomorphism to a direct sum of simpler bundles, and the dependence of this isomorphism on the scale can be deduced from \eqref{tract-change}. As the simplest example, we can define the \textit{standard cotractor bundle} as the dual bundle $\mathcal{T}^*M$ to $\mathcal{T}M$. Any choice of scale induces an isomorphism $\mathcal{T}^*M\cong (V^*\otimes L^*)\oplus L^*\oplus (E\otimes L^*)$ and the dependence of this isomorphism on $\xi_0$ can be deduced from \eqref{tract-change}.  

\subsection{The Schouten tensor}\label{4.4}
This is the last ingredient that we need in order to define the canonical tractor connection on $\mathcal{T}M$. 
The classical Schouten tensor of a Riemannian metric is a modification of the Ricci curvature (which contains the same information as the Ricci curvature), which has nicer behavior under conformal rescalings. There is a general version of the Schouten tensor associated to a Weyl structure for parabolic geometries and we introduce the version for exact Weyl structures on path geometries as an ad hoc definition. This tensor can be viewed as a one-form with values in $T^*M$ and hence as a $\binom{0}{2}$-tensor field, and we will describe it in terms of the splitting $TM=E\oplus V\oplus\iota(Q)$ induced by a choice of scale. 

As a preparation, we need some information on the torsion and the curvature of a Weyl connection. From the definition and the first condition in part (iii) of Theorem \ref{thm:main}, it follows that the restriction of the torsion $\tau$ to $E\times V$ coincides with $-\iota\circ\mathcal L:E\times V\to\iota(Q)$, so this is independent of the Weyl structure. The second condition in part (iii) of Theorem \ref{thm:main} implies that the first interesting component of the torsion is $\tau:H\times \iota(Q)\to H$, so in particular, we may look at the component mapping $E\times\iota(Q)$ to $V$. Given sections $\xi_1,\xi_2\in\Gamma (E)$, we get a section of $L(V,V)$ that sends $\eta\in\Gamma(V)$ to the $V$-component of $\tau(\xi_1,\iota(\mathcal L(\xi_2,\eta)))\in\Gamma(H)$ and we can form the (point-wise) trace of this to obtain a smooth function $\tr_V(\tau(\xi_1,\iota(\mathcal L(\xi_2,\ ))))$ on $M$.  

The Weyl connections on $TM$ are built up from connections on the bundles $E$, $V$ and $\iota(Q)$, so concerning their curvature, one should rather look at the curvatures of these individual connections. Starting from a scale, the connection on $E$ is flat by definition, while the connections on $V$ and $\iota(Q)$ are related by the isomorphism induced by $\iota\circ\mathcal L(\xi_0,\ ):V\to Q$. Hence their curvatures are also related by this isomorphism, so the main relevant quantity is the curvature $R^V\in\Omega^2(M,V^*\otimes V)$ of the linear connection $\nabla^V$ on $V$. One can form two traces of this curvature. The simpler one is the trace in the values, which is related to the curvature of the induced linear connection on $\Lambda^nV^*$ and hence (since the induced connection on $E$ is flat) of the induced linear connection on the line bundle $L$. Hence to avoid possible confusion, we will simply take the curvature $R^L\in\Omega^2(M)$ of the induced linear connection on $L$ as one of the basic inputs. For the second trace, we can restrict one of the form entries to $V$ and then form a Ricci-type contractions $\Ric^{\nabla^V}\in\Gamma(T^*M\otimes V^*)$, i.e.\ define  $\Ric^{\nabla^V}(\psi,\eta)=tr\,(R(\ ,\psi)(\eta))$ for $\psi\in\mathfrak X(M),\eta\in\Gamma(V)$. Armed with these observations, we can now define the Schouten tensor. 

\begin{definition}\label{def-Rho}
Let $(M,E,V)$ be a path geometry, $\xi_0\in\Gamma(E)$ a scale, $\tau$ the torsion of the induced Weyl connection, $R^L$ the curvature of the induced Weyl connection on $L$, and $\Ric^{\nabla^V}$ the Ricci-type contration of the curvature of the Weyl connection on $V$ as discussed above. Then we define (the components of the) Schouten tensor $\Rho\in\Gamma(T^*M\otimes T^*M)$ for $\xi,\xi_i\in\Gamma(E)$, $\eta,\eta_i\in\Gamma(V)$ and $\zeta_i:=\iota(\mathcal L(\xi_i,\eta_i))$ for $i=1,2$ as follows. 

(i) on $H\otimes H$:
\begin{align*}
      \Rho(\xi_1,\xi_2):&=-\tfrac 1 n \tr_V(\tau(\xi_1,\iota(\mathcal L(\xi_2,\ )))).\\
      \Rho(\xi,\eta)&=-2\Rho(\eta,\xi):=\tfrac{2}{2n+1}Ric^{\nabla^V}(\xi,\eta).\\
      \Rho(\eta_1,\eta_2):&=\tfrac{1}{n-1}Ric^{\nabla^V}(\eta_1,\eta_2). 
\end{align*}
(ii) on $H\otimes\iota(Q)\oplus\iota(Q)\otimes H$: 
\begin{align*}
    \Rho(\zeta_1,\eta)=-\Rho(\eta,\zeta_1):&=-R^L(\zeta_1,\eta).\\
\Rho(\zeta_1,\xi)+\Rho(\xi,\zeta_1):
&=(\nabla_{\xi_1}\Rho)(\eta_1,\xi)-(\nabla_{\eta_1}\Rho)(\xi_1,\xi).\\
\Rho(\zeta_1,\xi)-\Rho(\xi,\zeta_1):&=-R^L(\zeta_1,\xi).
\end{align*}
(iii) on $\iota(Q)\otimes\iota(Q)$: 
\begin{align*}
\Rho(\zeta_1,\zeta_2):&=(\nabla_{\zeta_1}\Rho)(\eta_2,\xi_2)-(\nabla_{\eta_2}\Rho)(\zeta_1,\xi_2)\\
&\quad -\Rho(\xi_1,\xi_2)\Rho(\eta_1,\eta_2)+\Rho(\eta_1,\xi_2)\Rho(\eta_2,\xi_1).
\end{align*}
\end{definition}

\begin{remark}\label{rem:Rho}
(1) While using this formally can be avoided by intricate computations, it is important to observe that these definitions reproduce the Schouten tensor provided by the theory of parabolic geometries. This ensures the validity of the formula for the tractor connection that we will derive. The basic way how this works is explained in part (3) of Remark \ref{rem:norm} above. One uses the relation between torsion, curvature and the Schouten tensor associated to a Weyl structure and the curvature $\kappa$ of the canonical Cartan connection together with information on $\kappa$ provided by the general theory. In principle, one could derive formulae for the components of $\Rho$ directly from the standard normalization condition on the Cartan curvature. Indeed, this is done for part (i) of Definition \ref{def-Rho}. However, there is a better approach, which leads to simpler expressions, since one can deduce additional information on the Cartan curvature using the so-called improved Bianchi identity. This is carried out in Lemma 14 of \cite{Guo_Fef}, showing vanishing of some components of $\kappa$, in particular of $\kappa|_{H\times H}$. In addition, by construction the Weyl connection $\nabla^E$ is flat. Knowing that some component of both $R$ and $\kappa$ vanishes, one concludes that the corresponding component of the tensor $\partial(\Rho)$ constructed from $\Rho$ vanishes, which leads to relations between components of $\Rho$. This produces the second and third formulae in (i) and the first and third formulae in (ii). For the remaining components, one (as expected) gets relations to traces of curvature and to (covariant) derivatives of previous components of $\Rho$ via the generalized Cotton-York tensor called $Y$ in Section 5.2 of \cite{Book}. 

(2) Observe that the Schouten tensor simplifies drastically for distinguished scales. As observed in Remark  \ref{rem:curvature}, the Ricci curvatures occurring in part (i) of Definition \ref{def-Rho} both can be expressed in terms of $d\alpha$, where $\alpha$ is the one-form characterized by $\Pi_E=\alpha\otimes\xi_0$. Hence for a distinguished scale, $\Rho$ vanishes on $V\times H$, which leads to further simplifications in the formulae in parts (ii) and (iii) of Definition \ref{def-Rho}. 
\end{remark}

\subsection{The tractor connection}\label{4.5}
We now have all the ingredients needed for explicit formulae for the tractor connection in the splitting determined by a scale, we just need a few further comments on notation. On the one hand, we have some natural contractions. For example for $\tau\in\Gamma(E^*\otimes L)$ and $\xi\in\Gamma(E)$, we write $\tau(\xi)\in\Gamma(L)$ for the obvious contraction. Slightly more complicated, for $\zeta=\iota(\mathcal L(\xi,\eta))\in\Gamma(\iota(Q))$ and $\tau$ as above, we obtain a well-defined section $\tau(\zeta):=\eta\otimes\tau(\xi)\in\Gamma(V\otimes L)$. On the other hand, for $\psi\in\mathfrak X(M)$ we can decompose $\Rho(\psi,\ )\in\Omega^1(M)$ as discussed in Section \ref{4.1} and we will write $\Rho^E(\psi)\in\Gamma(E^*)$, $\Rho^V(\psi)\in\Gamma(V^*)$ and $\Rho^Q(\psi)\in\Gamma(\iota(Q)^*)$ the components. Similarly as before, for a section $\nu\in\Gamma(V\otimes L)$, we can form $\Rho^Q(\psi)(\nu)\in\Gamma(E^*\otimes L)$. Using these observations, knowing that the Weyl structure, the Schouten tensor, and the splitting of the tractor bundle associated to a scale $\xi_0\in\Gamma(E)$ coincide with the objects coming from the general theory, Proposition 5.1.10 of \cite{Book} shows that the normal tractor connection $\nabla^{\mathcal T}$ is given by 
\begin{gather}
    \label{tract-xi}
    \nabla^{\mathcal T}_\xi\begin{pmatrix}\nu \\ \rho\\ \tau\end{pmatrix}_{\xi_0}=\begin{pmatrix}\nabla_\xi\nu \\ \nabla_\xi\rho+\tau(\xi)+\Rho(\xi,\nu)\\ \nabla_\xi \tau+\Rho^E(\xi)\otimes\rho-\Rho^Q(\xi)(\nu)\end{pmatrix}_{\xi_0}\\ 
    \label{tract-eta}
    \nabla^{\mathcal T}_\eta\begin{pmatrix}\nu \\ \rho\\ \tau\end{pmatrix}_{\xi_0}=\begin{pmatrix}\nabla_\eta\nu+\eta\otimes\rho \\ \nabla_\eta\rho+\Rho(\eta,\nu)\\ \nabla_\eta \tau+\Rho^E(\eta)\otimes\rho-\Rho^Q(\eta)(\nu)\end{pmatrix}_{\xi_0}\\
    \label{tract-zeta}
    \nabla^{\mathcal T}_\zeta\begin{pmatrix}\nu \\ \rho\\ \tau\end{pmatrix}_{\xi_0}=\begin{pmatrix}\nabla_\zeta\nu-\tau(\zeta) \\ \nabla_\zeta\rho+\Rho(\zeta,\nu)\\ \nabla_\zeta \tau+\Rho^E(\zeta)\otimes\rho-\Rho^Q(\zeta)(\nu)\end{pmatrix}_{\xi_0}
\end{gather}
Some facts that are known from the general theory are confirmed by these explicit formulae. For example, in the right hand side of \eqref{tract-eta} the two upper slots are independent of $\tau$. Hence the tractor connection actually induces a canonical partial linear connection in $V$-directions on the quotient bundle $\mathcal TM/\mathcal T^1M$. This reflects the fact that this quotient bundle is a relative tractor bundle. Indeed, it is easy to see that $\mathcal TM/\mathcal T^1M\cong TM/V\otimes (E^*\otimes L)$, and in Section \ref{4.2}, we have noted that $E^*\otimes L$ inherits a canonical partial connection in $V$-directions. On the other hand, as discussed in \cite{Relative}, $TM/V$ is a relative tractor bundle that carries a canonical partial connection in $V$-directions.

Similarly, the form of the top component of \eqref{tract-xi} shows that the subbundle $\mathcal T^0M$ (which corresponds to $\nu=0$) is invariant under derivatives in $E$-directions. Consequently, the tractor connection restricts to a canonical partial linear connection in $E$-directions on $\mathcal T^0M$. Again, this reflects the fact that $\mathcal T^0M$ is a relative tractor bundle (associated to a different parabolic subalgebra then the above), but this time this is not obtained from the general construction in \cite{Relative}.  

It is also easy to read off simple instances of the BGG machinery and the relative BGG machinery from the formulae for the tractor connection. Most easily, for $\rho\in\Gamma(L)$ we can restrict $\nabla\rho$ to a section $\nabla^E\rho\in\Gamma(E^*\otimes L)$ and consider the section (in obvious notation) $S(\rho):=(0,\rho,-\nabla^E\rho)_{\xi_0}\in\Gamma(\mathcal T^0M)$. Now from \eqref{tract-xi} it is obvious that $S(\rho)$ satisfies $\nabla^{\mathcal T}_\xi S(\rho)\in\Gamma(\mathcal T^1M)$ and is uniquely characterized by this condition as a section of $\mathcal T^0M$ projecting to $\rho$. This shows that $S(\rho)$ is independent of the choice of Weyl structure, so one has defined canonical differential operators $S:\Gamma(L)\to\Gamma(\mathcal T^0)$ and $D:\Gamma(L)\to\Gamma(E^*\otimes E^*\otimes L)$ with $D(\rho)$ defined as the restriction of $\nabla^{\mathcal T}S(\rho)$ to a section of $E^*\otimes\mathcal T^1M$. These are the first splitting operator and the first relative BGG operator determined by the relative tractor bundle $\mathcal T^0M$. 

One can also play a similar game with the relative tractor bundle $\mathcal TM/\mathcal T^1M$. Starting from $\nu\in\Gamma(V\otimes L)$, we can restrict $\nabla\nu$ to $\nabla^V\nu\in\Gamma(V^*\otimes V\otimes L)$ and then form an obvious contraction $\text{div}^V\nu\in\Gamma(L)$. We can then view $(\nu,-\tfrac1n\text{div}^V\nu,*)_{\xi_0}$ as a section $S(\nu)\in\Gamma(\mathcal TM/\mathcal T^1M)$. This has the property that the projection $D(\nu):=\nabla^V\nu-\tfrac1n\text{div}^V\nu$ of the restriction of $\nabla^{\mathcal T}S(\nu)$ to $V$-directions is trace-free and it is characterized by this property as a section of $\mathcal TM/\mathcal T^1M$. Thus we get a natural splitting operator $S:\Gamma(V\otimes L)\to\Gamma(\mathcal TM/\mathcal T^1M)$ and a relative BGG operator $D:\Gamma(V\otimes L)\to\Gamma((V^*\otimes V)_0\otimes L)$, where the subscript indicates the trace-free part.

\subsection{The dependence of the Schouten tensor on the scale}\label{4.6}
The only ingredient of the tractor connection, for which we have not yet clarified the behavior under a change of scale is the Schouten tensor. Understanding this is important for constructing invariant differential operators by direct computation. Moreover, it can be used to verify directly that the tractor connection is independent of the choice of scale and thus canonically associated to a path geometry. Our definition of the Schouten tensor involves components of the torsion and the curvature of the Weyl connections associated to a scale (which are formed using the decomposition $TM=E\oplus V\oplus\iota(Q)$ defined by the choice of scale). The dependence of the connection and the splitting on the scale is clarified in Propostions \ref{prop:Weyl-transf} and \ref{prop:L-trans}. Hence one can analyze the dependence of the relevant components of torsion and curvature by direct computation, which in turn can be used to analyze the dependence of the Schouten tensor. The necessary computations get quite involved however. 

Alternatively, one can use the fact that, as discussed in Section \ref{4.4} above, our definition of the Schouten tensor in Definition \ref{def-Rho} produces the Schouten tensor used in the general theory of parabolic geometries. This general theory also provides formulae for the dependence of the Schouten tensor on the scale, which are reasonably easy to handle. A bit of care is needed in these considerations, however. On the one hand, one needs the relation of the quantities $\Upsilon^E$, $\Upsilon^V$ and $\Upsilon^Q$ that are used in the general theory to $df$ from Remark \ref{rem4.2}. On the other hand, in the general theory, the Schouten tensor is interpreted as a one-form with values in the vector bundle $Q^*\oplus (E^*\oplus V^*)$. To get to our form of the Schouten tensor, one has to convert this into a one-form with values in $T^*M$, i.e.\ to a $\binom{0}{2}$-tensor field using the dual of the isomorphism $TM\cong E\oplus V\oplus Q$ provided by the Weyl structure. Once this is done, the computations are straightforward, so we don't reproduce them here but just collect the results.  

\begin{proposition}\label{prop:Rho-trans}
   Let $(M,E,V)$ be a (generalized) path geometry, $\xi_0\in\Gamma(E)$ a scale, and put $\hat\xi_0:=e^f\xi_0$ for a smooth function $f$. Then denoting quantities associated to $\hat\xi_0$ by hats, we get for $\xi,\xi_1,\xi_2\in\Gamma(E)$, $\eta,\eta_1,\eta_2\in\Gamma(V)$, $\zeta=\iota(\mathcal L(\xi,\eta))$ and so on, the following behavior under a change of scale. 

%\begin{equation}\label{Rho-transf}
    \begin{align*}
\hat{\Rho}(\xi_1,\xi_2)&=\Rho(\xi_1,\xi_2)-\tfrac12(\nabla_{\xi_1}df)(\xi_2)-\tfrac14 df(\xi_1)df(\xi_2).\\
   \hat{\Rho}(\xi,\eta)&=-2\hat\Rho(\eta,\xi)=\Rho(\xi,\eta)+(\nabla_\xi df)(\eta)-df(\xi)df(\eta)+df(\iota(\mathcal L(\xi,\eta))). \\
   \hat{\Rho}(\eta_1,\eta_2)&=\Rho(\eta_1,\eta_2)+(\nabla_{\eta_1}df)(\eta_2)-df(\eta_1)df(\eta_2).\\
   \hat{\Rho}(\zeta_1,\eta)&=\Rho(\zeta_1,\eta)+(\nabla_{\zeta_1}df)(\eta)+df(\eta_1)df(\iota(\mathcal L(\xi_1,\eta)))-df(\xi_1)df(\eta_1)df(\eta).\\
   \hat\Rho(\eta,\zeta_1)&=\Rho(\eta,\zeta_1)-\tfrac 1 2df(\xi_1)df(\eta_1)df(\eta)+df(\iota(\mathcal L(\xi_1,\eta)))df(\eta_1)\\
&\quad+(\nabla_{\eta}df)(\xi_1)df(\eta_1)+\tfrac 1 2 df(\xi_1)(\nabla_{\eta}df)(\eta_1)-(\nabla_{\eta}df)(\zeta_1).\\
\hat{\Rho}(\zeta_1,\xi)&=\Rho(\zeta_1,\xi)-\tfrac12(\nabla_{\zeta_1}df)(\xi)+\tfrac14df(\xi)df(\xi_1)df(\eta_1)-\tfrac12df(\xi)df(\zeta_1).\\
\hat{\Rho}(\xi,\zeta_1)&=\Rho(\xi,\zeta_1)+\tfrac12df(\xi)df(\xi_1)df(\eta_1)-\tfrac12df(\xi)df(\zeta_1)+(\nabla_{\xi_1} df)(\xi)df(\eta_1)\\
&\qquad+\tfrac12df(\xi_1)(\nabla_{\xi} df)(\eta_1)-(\nabla_{\xi} df)(\zeta_1).\\
\hat\Rho(\zeta_1,\zeta_2)&=\Rho(\zeta_1,\zeta_2)+(\nabla_{\zeta_1}df)(\xi_2)df(\eta_2)+\tfrac 1 2df(\xi_2)(\nabla_{\zeta_1}df)(\eta_2)-(\nabla_{\zeta_1}df)(\zeta_2)\\
-\tfrac 1 2&df(\xi_1)df(\xi_2)df(\eta_1)df(\eta_2)+\tfrac 1 2df(\xi_1)df(\eta_1)df(\zeta_2)+df(\xi_2)df(\eta_2)df(\zeta_1)-df(\zeta_1)df(\zeta_2).
\end{align*}
%\end{equation}
\end{proposition}

Let us give a simple example that illustrates how to use our results for an elementary construction of invariant differential operators. Consider the line bundle $L$ introduced in Section \ref{4.2}. Restricting derivatives by Weyl connections to directions in $E$, a choice of Weyl structure defines an operator $(\nabla_E)^2:\Gamma(L)\to\Gamma(S^2E^*\otimes L)$, which is explicitly given by 
$$
((\nabla_E)^2\rho)(\xi_1,\xi_2)=\nabla_{\xi_1}\nabla_{\xi_2}\rho-\nabla_{\nabla_{\xi_1}\xi_2}\rho.
$$  
The change of this expression caused by a change of Weyl structure is a straightforward computation using the first line of \eqref{n-trans-L} in both terms and \eqref{eq:ne-trans} in the second term. Using the fact that the expression $df(\xi_1)\nabla_{\xi_2}\rho$ is automatically symmetric in $\xi_1$ and $\xi_2$, one verifies that $$((\hat\nabla_E)\rho)(\xi_1,\xi_2)=((\nabla_E)^2\rho)(\xi_1,\xi_2)-\frac12(\nabla_{\xi_1}df)(\xi_2)\rho-\tfrac14 df(\xi_1)df(\xi_2)\rho.
$$
In view of the first formula in Proposition \ref{prop:Rho-trans}, we readily conclude that, denoting by $\Rho^{EE}$ the restriction of the Schouten tensor to $E\times E$, the expression $D\rho:=(\nabla_E)^2\rho-\Rho^{EE}\rho$ defines an invariant differential operator $\Gamma(L)\to\Gamma(S^2E^*\otimes L)$ of order two.


\end{document}